\documentclass{amsart}
\usepackage[english]{babel}
\usepackage[utf8]{inputenc}
\usepackage{amsmath, amssymb,epic,graphicx,mathrsfs,enumerate}
\usepackage[all]{xy}
\usepackage{color}
\usepackage{comment}
\usepackage{enumitem}
\usepackage{hyperref}
\usepackage[]{frontespizio}

\usepackage{amsmath}
\usepackage{amsthm}
\usepackage{amssymb}
\usepackage{latexsym}
\usepackage{epsfig}

\DeclareMathOperator{\aut}{Aut}

\DeclareMathOperator{\gl}{GL}

\newcommand{\alt}{\mathrm{Alt}}

\newtheorem{thm}{Theorem}

 \newtheorem{lemma}[thm]{Lemma}
\newtheorem{prop}[thm]{Proposition}

\numberwithin{equation}{section}

\renewcommand{\footnote}{\endnote}
\newcommand{\ignore}[1]{}\makeglossary

\begin{document}
	\bibliographystyle{amsplain}
\title[CIM-groups]{Finite groups in which every cyclic subgroup is the intersection of maximal subgroups}

\author{Andrea Lucchini}
\address{Andrea Lucchini\\ Universit\`a di Padova\\  Dipartimento di Matematica \lq\lq Tullio Levi-Civita\rq\rq\\ Via Trieste 63, 35121 Padova, Italy\\email: lucchini@math.unipd.it}

\begin{abstract}
We determine the structure of the finite groups with the property that  every cyclic subgroup is the intersection of maximal subgroups, comparing this property  with the one where all proper subgroups are intersections of maximal subgroups.
\end{abstract}

\maketitle

\section{Introduction}

A group $G$ is called an IM-group if every proper subgroup $1\leq H < G$  is the intersection of maximal subgroups of $G.$ An overview of the properties of IM-groups is presented in \cite[Section 3.3]{sh}. In 1970 Menegazzo gave a full characterization of the soluble IM-groups \cite{fm}.
Later is was conjectured that a finite IM-group must be soluble. Since the property of being an IM-group is inherited by quotients, the problem of the
solubility of finite IM-groups was first reduced to the simple case in \cite[Theorem 1.3]{bt}, where it was shown that a non-trivial normal subgroup of a finite IM-group is an IM-group. Then, using the classification of finite simple groups, it was proved in \cite{dmt} that every finite IM-group is soluble.
All these results together leads to the following conclusion: {\sl{a finite group $G$ is an IM-group if and only if $G$ contains a normal elementary abelian 
		Hall subgroup $N$ with elementary abelian factor such that every subgroup of $N$ is normal in $G$}} \cite[Theorem 3.3.10]{sh} (here we say that a finite abelian group is elementary abelian if its exponent is square-free).

Our aim is to study finite groups in which only subgroups with certain prescribed properties can be obtained as the intersection of maximal subgroups. In particular, we are interested in finite groups where every cyclic subgroup is the intersection of maximal subgroups. We will call such groups CIM-groups. Again, using the classification of finite simple groups, it can be proved that finite CIM-groups are soluble. More precisely, the following result holds:

\begin{thm}\label{supersu} A finite CIM-group $G$ is supersoluble and metabelian. More precisely there exists a finite abelian group $H$, a family of irreducible $H$-modules $V_1,\dots, V_r$ of prime order and positive integers $\delta_1,\dots,\delta_r$ such that
	$G\cong (V_1^{\delta_1} \times \dots \times {V_r^{\delta_r})}\rtimes H.$
	Moreover $|V_1|,\dots,|V_t|$ are pairwise different and none of them divides $|H|.$
\end{thm}

A semidirect product with the properties described in the previous statement is not necessarily a CIM-group. To state a precise and complete characterization of CIM-groups, it is necessary to introduce some additional notations. Given $h \in H$,  define $M_H(h)$ as the intersection of the maximal subgroups of $H$ that contain $h$. 

\begin{thm}\label{completo}
	A semidirect product $G = (V_1^{\delta_1} \times \dots \times {V_r^{\delta_r})}\rtimes H$ with the properties described in the statement of Theorem \ref{supersu} is a CIM-group if and only if
	\ $\bigcap_{j\in J_G(h)}C_{M_H(h)}(V_j)=\langle h\rangle$ for every $h\in H$
	with $\langle h\rangle \neq H.$
\end{thm}

One of the consequences of the previous theorem is that there exist CIM-groups that are not IM-groups. Let us denote by $F(G)$ the Fitting subgroup of $G.$ From the results mentioned above, if $G$ is an IM-group, then both $F(G)$ and $G/F(G)$ are elementary abelian. In the case of CIM-groups, it remains true that $F(G)$ is elementary abelian, while $G/F(G)$ is abelian but not necessarily elementary abelian. In fact, the following result holds:
\begin{thm}\label{qab}
	For every finite abelian group $A$, there exists a finite  CIM-group $G$ such that $G/F(G)$ is isomorphic to $A.$
\end{thm}
The property of being an IM-group is inherited by quotients (see \cite[Lemma 3.3.7]{sh}). From Proposition \ref{qab}, it follows that this is no longer true for CIM-groups (in fact, every finite abelian group appears as a quotient of some CIM-group, but only elementary abelian groups are CIM-groups). Furthermore, normal subgroups of IM-groups are IM-groups (see \cite[Lemma 3.3.7]{sh}) and IM-groups are T-groups (see \cite[Lemma 3.3.9]{sh}), that is, groups in which the normality relation is transitive. The same results also hold for finite CIM-groups (see Propositions \ref{tcim} and \ref{ns}), but while in the case of IM-groups these properties can be easily proven, for finite CIM-groups we are only able to deduce these results from Theorems \ref{supersu} and \ref{completo}, thus implicitly relying on the classification of finite simple groups.

While the class of finite CIM-groups strictly contains that of IM-groups, the same does not occur when considering the class of finite groups in which every abelian subgroup is the intersection of maximal subgroups. We call such groups AIM-groups. As a consequence of Theorem \ref{supersu} and \ref{completo}, we obtain:

\begin{thm}\label{aim}
	A finite group $G$ is an IM-group if and only if it is an AIM-group.
\end{thm}

Denote by $P_{IM}(G)$ the probability that a proper subgroup of $G$ is the intersection of maximal subgroups. We observe that requiring  $P_{IM}(G)$ to be sufficiently close to 1 does not guarantee that $G$ is an IM-group.
Indeed the following holds:

\begin{thm}\label{solouno}
	For every positive integer $n$ there exists a finite group $G$ containing at least $n$ subgroups and in which there is only one proper subgroup which is not the intersection of maximal subgroups.
\end{thm}

From the proof of the previous theorem, it also follows that requiring all subgroups of $G$ whose order is a prime power to be intersections of maximal subgroups is not sufficient to guarantee that $G$ is an IM-group.


Another property of a finite group that can be considered is the following: all its Sylow subgroups are intersection of maximal subgroups. One may wonder if this property is strong enough to imply the solubility of finite groups that satisfy it. In Section \ref{ultimo} we prove that the answer is negative. 

\subsection*{Acknowledgements} I thank Chiara Tamburini for drawing my attention to this topic and for her useful suggestions.

\section{CMI-groups are soluble}

The aim of this section is to prove that a finite CIM-group is soluble. A crucial ingredient of our proof is the following result, which depends on the classification of the finite simple groups.   
\begin{thm}\label{almost}\cite[Theorem 3]{indip}
	Let $S$ be a non-abelian finite simple group. Then there exist non-commuting elements $s,x \in S$ such that, for each almost simple group $X$ with socle $S$, the intersection of the maximal subgroups of $X$ containing $s$ contains also $x$.
\end{thm}  

\begin{thm}\label{solub} A finite CIM-group is soluble.
\end{thm}     
\begin{proof}Let $G$ be a finite CIM-group. In particular the identity subgroup is the intersection of the maximal subgroups of $G$ and therefore the Frattini subgroup of $G$ is trivial. This implies that the generalized Fitting subgroup $F^*(G)$ of $G$ is isomorphic to the direct product of minimal normal subgroups of $G$. Let $M$ be an abelian minimal normal subgroup of $G.$ A maximal subgroup of $G$ either contains $M$ or has with $M$ a trivial intersection. In particular if $1\neq m \in M$, then all the maximal subgroups of $G$ containing $m$ contain the entire subgroup $M$, and therefore, since $G$ is a CIM-group,
	it must be $M=\langle m\rangle.$ So all the abelian minimal normal subgroups of $G$ are cyclic.
	Assume that $F^*(G)$ is soluble. Then $F^*(G)=M_1 \times \dots \times M_r,$ where $M_i$ is normal and cyclic for every $1\leq i\leq r.$ Since $C_G(F^*(G))\leq F^*(G),$ we have that
	$$\frac{G}{F^*(G)} = \frac{G}{C_G(F^*(G))} = \frac{G}{\bigcap_{1\leq i \leq r}C_G(M_i)} \leq \prod_{1\leq i \leq r}\frac{G}{C_G(M_i)}
	\leq \prod_{1\leq i \leq r}\aut(M_i),$$
	but then $G/F^*(G)$ is abelian and $G$ is metabelian. 
	
	Thus we may assume that $G$ contains a non-abelian minimal normal subgroup, say $N.$
	So $N\cong S^t$ where $S$ is
	a non-abelian finite simple group.
	Let $S_i$ be the
	subgroup of $N$ consisting of the $t$-uples with 1 in all but the $i$-th
	component: $S \cong S_i$ and $ N \cong S_1 \times \dots \times S_t$. 
	Choose $s, x \in S$ as described in Theorem \ref{almost} and let $g=(s,1,\dots,1) \in S_1.$ Consider a maximal subgroup $M$ of $G$ with $g\in M$ and $N\not\leq M.$ Since $G=NM,$ we must have that $N\cap M$ is a maximal $M$-invariant subgroup subgroup of $N.$
	Since $g\in M\cap N,$ it must be $M\cap N\neq 1.$ Consider the homomorphism $\phi: N_G(S_1)\to \aut(S)$ induced by the conjugation action of $N_G(S_1)$ on $S_1$ and let $X:=\phi(N_G(S_1)).$ Then $X$ is an almost simple group with socle $\phi(S_1)$ isomorphic to $S.$ 
	Let $\pi _i$ be the projection of the group $N=S_1 \times \dots \times S_t$ onto
	the $i$-th component $S_i$. The possibilities for $M\cap N$  are described in
	the proof of the O'Nan-Scott theorem \cite{ons}. Since $M\cap N\neq 1$, then one of the following occurs:
	\begin{enumerate}
		\item  $\pi _1(N\cap M)=S_1$. Since $N\not\leq M$ it must be $t>1.$ Moreover there exists a partition $I_1, \dots, I_l$ of the set
		$\{1,\dots,t\}$ such that $M\cap N=D_1 \times \dots \times D_l$, where $D_i$ is a full
		diagonal subgroup of $\prod_{j \in I_i}S_j$ (i.e. the restriction to 
		the subgroup $D_i$ of the
		projection  $\pi_k: \prod_{j \in I_i}S_j \rightarrow S_k$ is an isomorphism 
		for any $k \in I_i$). But $(s,1,\dots,1)$ cannot be contained in such product of diagonal subgroups, so we may exclude this first possibility.
		\item $\pi_1(N\cap M)=R_1 < S_1$. In this case
		$N\cap M= R_1\times \dots \times R_m$ with $R_i\cong R_1$ for every $1\leq i\leq t$ and
		$R_1$ is a maximal $N_M(S_1)$-invariant subgroup of $S_1.$
		Let $Y=\phi(N_M(S_1)).$ From $NM=G$ it follows
		$SY=X.$ Moreover $S\cap Y=\phi(S_1\cap M)=\phi(R_1).$ It follows that
		$Y$ is a maximal subgroup of $X$ containing $s=\phi(g).$ But then $x \in \phi(R_1)$ and  therefore $(x,1,\dots,1)\in M.$ 
	\end{enumerate}
	We have so proved that every maximal subgroup of $G$ containing $g$ contains also $(x,1,\dots,1),$ in contradiction with the assumption that $G$ is a CIM-group (notice that the fact that $s$ and $x$ do not commute, implies $(x,1,\dots,1)\not\leq \langle g\rangle).$
\end{proof} 

\section{Soluble CIM-groups}

Let $G$ be a soluble finite CIM-group. Since the Frattini subgroup of $G$ is trivial, the Fitting subgroup $F(G)$ of $G$ is a direct product of minimal normal subgroups of $G,$ and is complemented in $G$. Let $H$ be a complement of $F(G)$ in $G.$ Then
$$G=(V_1^{\delta_1} \times \dots \times {V_r^{\delta_r})}\rtimes H,$$
where $\delta_1, \dots, \delta_r$ are positive integers, and $V_1,\dots,V_r$ are
irreducible $H$-modules that are pairwise non-$H$-isomorphic.

The maximal subgroups of $G$ are of two possible types:
\begin{enumerate}
	\item $M=F(G)K$ with $K$ a maximal subgroup of $H.$
	\item $M=W H^u$ where $F(G)=U\oplus W$ with
	$U$ and $W$  $H$-submodules of $F(G),$ $U$  irreducible and $u\in U.$ We will use the notation $M(W,U,u)$ for such a subgroup.
\end{enumerate}

In particular if $N$ is a minimal normal subgroup of $G$, then $N\leq F(G).$ Moreover if a maximal subgroup of $G$ contains a non-trivial element of $N$, then it contains $N$ and therefore, since $G$ is a CIM-group, $N$ must be cyclic. Thus for every $1\leq i\leq r,$ the order of $V_i$ is a prime, say $p_i.$

Since $C_G(F(G))\leq F(G),$ we must have
$\cap_{1\leq i\leq r} C_H(V_i)=C_H(F(G))=1.$ Hence
$$H\cong \frac{H}{\bigcap_{1\leq i\leq r}C_H(V_i)}\leq \prod_{1\leq i\leq r}\frac{H}{C_H(V_i)}\leq \prod_{1\leq i\leq r}\aut(C_{p_i}),$$
and therefore $H$ is abelian. 

\begin{lemma}\label{coprime}
	If $i\neq j$, then $p_i\neq p_j.$
\end{lemma}
\begin{proof}
	Assume, by contradiction, $p_i=p_j.$ The Fitting subgroup of $G$ contains two minimal normal subgroups of $G$, say $A$ and $B$, with $A\cong_H V_i$ and $B\cong_H V_j.$ Let $a\in A,$ $b\in B$ with $|a|=|b|=p_i$ and let $g=ab.$  All the maximal subgroups of $G$ of type (1) contains $\langle a, b\rangle.$ So the assumption that $\langle g\rangle$ is the intersection of maximal subgroups of $G$, implies that there must exist a maximal subgroup $M(W,U,u)$ of type $(2)$ which contains $g$ but neither $a$ nor $b$. This implies $F(G)=A\oplus W=B\oplus W,$ and therefore $A\cong_HU\cong_H B$ in contradiction with $V_i \not \cong_H V_j.$
\end{proof}

\begin{lemma}\label{coprimeh}
	$(|H|,p_1\cdots p_r)=1.$
\end{lemma}
\begin{proof}
	Suppose, by contradiction, that, for some $i\in \{1,\dots,r\},$ $H$ contains an element $y$ of order $p_i.$ Let $x\in F(G)$ of order $p_i$, which implies in particular $\langle x \rangle \cong_H V_i$ and consider $g=xy.$ 
	A maximal subgroup of $G$ of type (1) which contains $g$ contains $\langle x, y\rangle.$ So the assumption that $\langle g\rangle$ is the intersection of maximal subgroups of $G$, implies that there must exist a maximal subgroup $M=M(W,U,u)$ of type $(2)$ which contains $g$ but not $x.$ In particular $x\notin W$ so $U\cong_H \langle x\rangle.$ Since $\langle x, y\rangle$ is an abelian group of order $p_i^2,$ we have $[y,u]=[y,x]=\{0\}.$ But then $y \in H^u$ and therefore $M$ contains $y$ and consequently also $x$, a contradiction.
\end{proof}


\begin{lemma}\label{brutto}
	Let $x=yh\in G$ with $y\in F(G)$ and $h \in H.$ Let $\tilde J_G(x)$ be the set of the elements $i\in \{1,\dots,r\}$ such that $h$ centralizes $V_i$ and either $p_i$ is coprime with $|y|$ or $\delta_j>1.$
	Moreover let $M_H(h)$ be the intersection of the maximal subgroups of $H$ containing $h$ if $\langle h \rangle \neq H,$ $M_H(h)=H$ otherwise. The following are equivalent:
	\begin{enumerate}
		\item $\langle x\rangle$ is the intersection of maximal subgroups of $G;$
		\item 	 $\bigcap_{j\in \tilde J_G(x)}C_{M_H(h)}(V_j)=\langle h \rangle.$
	\end{enumerate}
\end{lemma}
\begin{proof}
	If $[h,V_i]\neq \{0\},$ then $vh$ is conjugate to $h$ for every $v\in  V_i,$ so, up to conjugation, we may assume $[y,h]=0.$  By Lemma \ref{coprimeh}, $\langle x\rangle=\langle y, h\rangle.$ Moreover, by Lemma \ref{coprime}, $\langle y\rangle$ is a direct sum of minimal normal subgroups of $G$ contained in $F(G)$ and this implies that, for every $z\in F(G),$ $\langle y\rangle H^z$ is the intersection of maximal subgroups of $G$. In particular if $\langle h\rangle = H$, then $\langle x\rangle=\langle y\rangle H$ is the intersection of maximal subgroups, so we may assume $\langle h\rangle \neq H.$
	Let $M=M(W,U,u)$ be a maximal subgroup of $G$ containing $h$. Since $y$ and $h$ commute and have coprime orders,
	$x\in M$ if and only if $y \in W$ and $h \in h^uW.$ The second condition is equivalent to $[h,u]\in U\cap W=\{0\}.$ Let $\Omega$ be the set of  $u\in F(G)$ such that $[u,h]=0$ and
	$\langle u \rangle$ is an $H$-submodule of $F(G)$ complementing a maximal $H$-submodule $W$ containing $\langle y\rangle.$ 
	Notice that $F(G)M_H(h)$ is the intersection of the maximal subgroups of type (1) of $G$ containing $x$. Consequently
	$$\langle y \rangle (M_H(h))^u= F(G)(M_H(h))^u\cap \langle y \rangle H^u$$  can be obtained as intersection of maximal subgroups of $G$ for every $u\in \Omega.$ It follows that the intersection of the maximal subgroups of $G$ containing $x$ coincides with $$X:=\bigcap_{u\in \Omega}\langle y\rangle (M_H(h))^u.$$
	Notice that if $u\in \Omega,$ then $\langle u\rangle \cong_H V_j$ for some $j\in \tilde J(x).$ Hence
	$$C:=\bigcap_{j\in \tilde J(x)}C_{M_H(h)}(V_j)=\bigcap_{u \in \Omega}C_{M_H(h)}(u).$$
	In particular if $c \in C,$ then $c\in X,$ and so (1) implies (2). Conversely assume $x\in X.$ Then $x=y^iz$ for some $i \in \mathbb Z$ and $z\in H.$ In particular,
	for every $u\in \Omega,$ $x\in \langle y\rangle (M_H(h))^u \cap \langle y\rangle M_H(h)$ and therefore $[z,u]\in \langle y\rangle \cap \langle u\rangle=\{0\}.$ So $z\in C$, hence (2) implies (1).
\end{proof}

\begin{proof}[Proof of Theorems \ref{supersu} and \ref{completo}]Theorems \ref{supersu} follows from Theorem \ref{solub} and Lemmas \ref{coprime} and \ref{coprimeh}. The proof of Theorem \ref{completo} follows from Lemma \ref{brutto} combined with the following remarks. Given $h\in H,$ $J_G(h)\subseteq \tilde J_G(fh)$ for every $f\in F(G).$ Moreover $J_G(h)=\tilde J_G((\prod_{j\in J_G(h)}u_j)h),$ where for every $j\in J_G(h)$ we choose  $u_i\in V_i.$ 
\end{proof}

Before using Theorem \ref{completo} to further investigate the properties of CIM-groups, let us show some examples that can make it clearer.
Let $H=	\langle x\rangle$ be a cyclic group of order 8
and consider the following $H$-module: $V_1$ has order 5 and $C_H(V_1)=\langle x^4\rangle,$ $V_2$ has order 17 and $C_H(V_1)=\{1\}.$ We consider three semidirect products constructed using $H$ and the modules $V_1$ and $V_2.$
\begin{enumerate}
	\item Let $G_1=(V_1^2\times V_2^2) \rtimes H.$ If $h$ has order 4, then $M_h(h)=\langle h\rangle$; if $h=x^4,$ then $J_G(h)=\{1\}$, $M_H(h)=\langle x^2\rangle$ and $C_{M_H(h)}(V_1)=\langle h\rangle;$ if $h=1,$ then $J_G(h)=\{1,2\}$, $M_H(h)=\langle x^2\rangle$ and $C_{M_H(h)}(V_1)\cap C_{M_H(h)}(V_2)=\langle 1\rangle.$
	Hence, by Theorem \ref{completo}, $G_1$ is a CIM-group.
	\item Let $G_2=(V_1^2\times V_2)\rtimes H.$ In this case
	$J_G(1)=\{1\}$, $M_H(1)=\langle x^2\rangle$ and $C_{M_H(1)}(V_1)=\langle x^4\rangle.$ By Lemma \ref{brutto} and its proof, if $g$ has order 17 or 85, then the intersection of the maximal subgroups of $G_2$ containing $|g|$ has order $2|g|.$
	\item Let $G_3=(V_1\times V_2^2)\rtimes H.$ In this case if $h=x^4,$ then
	$J_G(h)=\emptyset$ and $M_H(h)=\langle x^2\rangle$. 
	If $g$ has order 10, then the intersection of the maximal subgroups of $G_3$ containing $|g|$ has order $20.$
\end{enumerate}

\begin{prop}\label{tcim}If $G$ is a finite CIM-group, then $G$ is a $T$-group.
\end{prop}
\begin{proof}By Theorem \ref{supersu}, $G=(V_1^{\delta_1}\times \dots \times V_r^{\delta_r})\rtimes H$ , with $H$ abelian and $|V_1|,\dots,|V_r|$ of prime order. For $1\leq i\leq r,$ let $W_i=V_i^{\delta_i}.$ If $h\in H$ and $[V_i,h]\neq \{0\},$ then
	$h^{W_i}=W_ih$ and from this it follows that a subgroup $X$ of $G$ is normal if and only if $W_i\leq X$ for every index $i$ such that $V_i$ is not centralized by $X.$ Now assume $X\unlhd Y \unlhd G.$ If $[V_i,X]\neq \{0\}$ then $[V_i,Y]\neq \{ 0\}$ and therefore $W_i\leq Y.$ But then, if $x\in X$ and $[V_i,x]\neq \{0\},$ then $W_ix=x^{W_i}\subseteq x^{Y_i}\subseteq X$. Thus $W_i\leq X$ whenever $[V_i,X]\neq \{0\}$ and therefore $X$ is normal in $G.$
\end{proof}

\begin{prop}\label{ns}
	Let $G$ be a finite CIM-group. If $N$ is a normal subgroup of $G$, then $N$ is a CIM-group.
\end{prop}
\begin{proof}
	Since $G$ is CIM, $G=(V_1^{\delta_1}\times \dots \times V_r^{\delta_r})\rtimes H$ as is the statement of Theorem \ref{supersu}. Let $N$ be a normal subgroup of $G$ and let $J$ be the set of the indices $j$ such that $[V_j,N]\neq \{0\}.$ Then $N=TK,$ where
	$K\leq H,$ $T\leq  V_1^{\delta_1}\times \dots \times V_r^{\delta_r}$ 
	and $V_j^{\delta_j}\leq T$ for every $j\in J.$ If $x \in K$ and $\langle x\rangle \neq K$ then
	$M_K(x)\leq M_H(x)$ and therefore
	$$\bigcap_{j\in J_N(x)}C_{M_K(x)}(V_j) = \bigcap_{j\in J_G(x)}C_{M_K(x)}(V_j) \leq \bigcap_{j\in J_G(x)}C_{M_H(x)}(V_j) 
	$$
	where the first equality follows from the fact that if $j\in J_G(x)\setminus J_N(x),$ then
	$V_j^{\delta_j} \not\leq N$ and therefore $[V_j,K]=\{0\}$ and $C_{M_K(x)}(V_j)=M_K(x).$ Since $G$ is a CIM-group, by Theorem \ref{completo} we have that $\bigcap_{j\in J_G(k)}C_{M_H(x)}(V_j)=\langle x\rangle.$ But then
	$\bigcap_{j\in J_N(x)}C_{M_K(x)}(V_j)=\langle x\rangle$ and so, again by Theorem \ref{completo}, $N$ is a CIM-group.
\end{proof}

\begin{proof}[Proof of Theorem \ref{qab}]
	Let $\Delta$ be the set of the pairs $(a,b)\in A^2,$ such that $A\neq \langle a\rangle$ and $b\notin \langle a\rangle$. For every $\delta=(a,b)\in \Delta$ there exists a subgroup $C$ of $A$ such that $a\in C, b\notin C$ and $G/C$ is cyclic. Let $\Omega_\delta$ be the set of primes $p$ such that $|G/C|$ divides $p-1.$ By  the Dirichlet prime number theorem, $\Omega_\delta$ is infinite. If $p\in \Omega_\delta$, then $G/C$ is isomorphic to a subgroup of the automorphism group of a cyclic group of order $p$ and therefore there exists an $A$-module $V(\delta,p)$ such that $|V(\delta,p)|=p$, $a\in C_A(V(\delta,p)),$ $b\notin C_A(V(\delta,p)).$ Since $\Delta$ is finite and $\Omega_\delta$ is infinite for every choices of $\delta$, we may find a set of primes $\{p_{\delta}\}_{\delta \in \Delta}$ with the property that $p_{\delta_1}\neq p_{\delta_2}$ whenever $\delta_1\neq \delta_2.$ It follows from Theorem \ref{completo} that
	$G:=\left(\prod_{\delta \in \Delta}V(\delta, p_\delta)^2\right)\rtimes A$ is a CIM-group.
\end{proof}

\begin{proof}[Proof of Theorem \ref{aim}] Clearly a finite IM-group is an AIM-group. Conversely, assume that $G$ is an AIM-group. Then in particular $G$ is a CIM-group, so by Theorems \ref{supersu} and \ref{completo}, $G=F(G)\rtimes A$, where $A$ is abelian, $F(G)$ is elementary abelian and all the subgroups of $F(G)$ are normal in $G$. Since $F(G)$ is the intersection of maximal subgroups of $G$, the Frattini subgroup of $A$ must be trivial, which is equivalent to say that $A$ is elementary abelian. But then $G$ is an IM-group by the characterization of IM-groups mentioned in the introduction.
\end{proof}

\begin{proof}[Proof of Theorem \ref{solouno}]Let $H$ be a cyclic group of order 4 and let $V_1$ and $V_2$ be two irreducible $H$-modules $V_1$ and $V_2$ with $|V_1|=5,$
	$|V_2|=13$ and $C_H(V_1)=C_H(V_2)=\{1\}.$ Consider the semidirect product $G:=(V_1^{\delta_1}\times V_2^{\delta_2})\rtimes H,$ with $\delta_1,\delta_2\geq 2.$ Let $F=V_1^{\delta_1}\times V_2^{\delta_2}$ and let $K$ be a proper subgroup of $G$. Consider $N=K\cap F.$ Then $N\cong V_1^{\eta_1}\times V_2^{\eta_2}$ with $\eta_1$ and $\eta_2$ non-negative integer and $K/N$ is a cyclic subgroup of order dividing 4 of the factor group $G/N \cong (V_1^{\delta_1-\eta_1}\times V_2^{\delta_2-\eta_2})\rtimes H.$ It follows from Lemma \ref{brutto} that $K/N$ is the intersection of maximal subgroups except when $|K/N|=1$ and $\delta_1-\eta_1=\delta_2-\eta_2=0.$ This is equivalent to say that $F$ is the unique proper subgroup of $G$ which is not the intersection of maximal subgroups.
\end{proof}

\section{Sylow subgroups as intersection of maximal subgroups}\label{ultimo}

In this section we construct a finite group $G$ which is not soluble but all whose Sylow subgroups are the intersection of maximal subgroups.

Let $S=\alt(5)$ be the alternating group of degree 5. Let $a=(1,2)(3,4)$ and $b=(1,3,5).$ Then $S$ has two faithful irreducible representations, of degree, respectively, 3 and 5, over the field of order 11, corresponding to the homomorphisms $\phi_1: S\to \gl(3,11)$ and $\phi_2: S\to \gl(5,11)$ defined by setting:
$$\phi_1(a)=\begin{pmatrix}-1&0&0\\0&-1&0\\4&4&1\end{pmatrix},\quad \phi_1(b)=\begin{pmatrix}0&1&0\\0&0&1\\1&0&0\end{pmatrix},$$
$$\phi_2(a)=\begin{pmatrix}1&0&0&0&0\\0&0&1&0&0\\0&1&0&0&0\\0&0&0&1&0\\-1&-1&-1&-1&-1\end{pmatrix}, \quad \phi_2(b)=\begin{pmatrix}0&1&0&0&0\\0&0&0&1&0\\0&0&0&0&1\\1&0&0&0&0\\-1&-1&-1&-1&-1\end{pmatrix}.$$
In correspondence we can find two irriducible $S$-modules, of order, respectively, $11^3$ and $11^5.$
Let $G=(V_1\times V_2)\rtimes S.$

The normal subgroup $V_1\times V_2$ is the unique Sylow 11-subgroup of $G$ and it is the intersection of maximal subgroups of $G$ since the Frattini subgroup of $G/(V_1\times V_2)\cong S$ is trivial.

A Sylow 3-subgroup of $G$ is conjugate  to the subgroup $\langle (1,2,3) \rangle$ of $S$. Notice that $\langle (1,2,3)\rangle = M_1\cap M_2$ with $M_1, M_2$ two maximal subgroups of $S$ isomorphic to $\alt(4).$ It follows that  $\langle (1,2,3) \rangle = V_1S \cap V_2S \cap (V_1\times V_2)M_1\cap (V_1\times V_2)M_2.$

A Sylow 5-subgroup of $G$ is conjugate  to the subgroup $\langle (1,2,3,4,5) \rangle$ of $S$. There exists a unique maximal subgroup $M$ of $S$ containing $(1,2,3,4,5)$. This subgroup has order 10 and there exists $v \in V_2$ such that $C_M(v)=\langle (1,2,3,4,5) \rangle$. Then
$\langle (1,2,3,4,5) \rangle = V_1S \cap V_2S \cap V_1M\cap V_1M^v.$

A Sylow 2-subgroup of $G$ is conjugate  to the subgroup $\langle (1,2)(3,4), (1,3)(2,4) \rangle$ of $S$. There exists a unique maximal subgroup $M$ of $S$ containing $(1,2)(3,4)$ and $(1,3)(2,4).$ It has order 12 and there exists $w \in V_2$ such that $C_M(w)=\langle  (1,2)(3,4), (1,3)(2,4) \rangle$. Then
$\langle (1,2)(3,4), (1,3)(2,4) \rangle = V_1S \cap V_2S \cap V_2M\cap V_2M^w.$


\begin{thebibliography}{99}
	\bibitem{bt} M. Bianchi, M. C. Tamburini Bellani, Sugli IM-gruppi finiti e i loro duali, Istit. Lombardo Accad. Sci. Lett. Rend. A 11 (1977), 429-436.
	\bibitem{indip} S. D. Freedman, A. Lucchini, D. Nemmi and C. M.  Roney-Dougal,  Finite groups satisfying the independence property, Internat. J. Algebra Comput. 33 (2023), no. 3, 509--545.
	\bibitem{dmt} L. Di Martino and  M. C. Tamburini, 
	On the solvability of finite IM-groups, 
	Istit. Lombardo Accad. Sci. Lett. Rend. A 115 (1981), 235–242 (1984).
\bibitem{ons} M.	Liebeck, C. Praeger and J. Saxl,  On the O'Nan-Scott theorem for finite primitive permutation groups, J. Austral. Math. Soc. Ser. A 44 (1988), no. 3, 389–396.
\bibitem{fm} F. Menegazzo, Gruppi nei quali ogni sottogruppo è intersezione di sottogruppi massimali, Atti Accad. Naz. Lincei Rend. Cl. Sci. Fis. Mat. Nat. (8) 48 (1970), 559–562.
\bibitem{sh} R. Schmidt, Subgroup lattices of groups, De Gruyter Expositions in Mathematics, 14. Walter de Gruyter \& Co., Berlin, 1994.
	\end{thebibliography}
\end{document}